\documentclass[a4paper, 11pt]{article}
\usepackage{amsmath, amsthm, amssymb, fullpage, graphicx, centernot, url}

\newtheorem{theorem}{Theorem}
\newtheorem{lemma}[theorem]{Lemma}
\theoremstyle{remark}
\newtheorem*{remark}{Remark}

\newcommand*{\pc}[2][]{#1{p}_\mathrm{c}^{#2}}
\newcommand*{\future}[1][0]{(#1,\infty)}
\newcommand*{\prob}[1]{\mathbb{P}(#1)}
\newcommand*{\probz}[1]{\mathbb{P}_{\mathbb{Z}}(#1)}
\newcommand*{\probb}[1]{\mathbb{P}\bigl(#1\bigr)}
\newcommand*{\probbz}[1]{\mathbb{P}_{\mathbb{Z}}\bigl(#1\bigr)}
\newcommand*{\onab}[3][1]{\{#3\}_{[x_{#1},x_{#2}]}}
\newcommand*{\onabz}[3][1]{\{#3\}_{[#1,#2]}}
\newcommand*{\onfuture}[2][0]{\{#2\}_{\future[#1]}}
\newcommand{\from}{\mathbin{\leftarrow}}
\renewcommand{\to}{\mathbin{\rightarrow}}
\newcommand{\triple}{\to\!\bullet\!\from\bullet}
\newcommand{\nfrom}{\mathbin{\centernot\leftarrow}}
\newcommand*{\go}[1][]{\vec{\bullet}_{#1}}
\newcommand*{\come}[1][]{\reflectbox{\ensuremath{\vec{\bullet}}}_{#1}}
\newcommand*{\stay}[1][]{\dot{\bullet}_{#1}}
\newcommand{\ie}{i.e.\ }
\newcommand{\eg}{e.g.\ }
\newcommand*{\bbr}[1]{\Bigl(#1\Bigr)}
\newcommand*{\lem}[1]{Lemma~\ref{#1}}
\newcommand*{\thm}[1]{Theorem~\ref{#1}}
\DeclareMathOperator{\rev}{rev}

\title{The ballistic annihilation threshold is positive}
\author{John Haslegrave\thanks{Mathematics Institute, University of Warwick, Coventry, UK}}

\begin{document}
\maketitle

\begin{abstract}
In the ballistic annihilation process, particles on the real line have
independent speeds symmetrically distributed in $\{-1,0,+1\}$ and
are annihilated by collisions. 
It is widely believed that there is a phase transition at $p=\pc{}=0.25$
between regimes where every particle is eventually annihilated and 
where some particles survive forever, where $p$ is the proportion
of stationary particles. It is easy to see that some 
particles survive if $p>0.5$, and rigorous proofs giving
better upper bounds on $\pc{}$ have recently appeared.
However, no nontrivial lower bound on $\pc{}$
was previously known. We prove that $\pc{}\geq 0.21699$,
and give a comparable bound for a discretised version.


\end{abstract}

\section{Introduction}
In the ballistic annihilation process, particles start as a Poisson point process on the real line. 
Each particle independently is either stationary (with probability $p$) or moving at unit 
speed left or right (each with probability $\frac{1-p}2$). When two particles collide, they annihilate each other.

This process, in a variety of forms, was studied by physicists in the 1980s and 1990s. Early work (\eg \cite{EF85,KS88}) 
considered the case where only two speeds are permitted; subsequently continuous speed distributions 
and more general discrete distributions were studied (\cite{BRL93,KRL95,DRFP,Red97}), with the latter showing more 
interesting behaviour. In the canonical three-speed case described above,
Krapivsky, Redner and Leyvraz \cite{KRL95} postulated the existence of a critical probability $\pc{}$, such that for $p<\pc{}$ 
every particle is eventually annihilated, whereas for $p>\pc{}$ some particles survive forever.
Based on a heuristic derived from considering the rate at which different types of collisions might be expected to occur,
they conjectured that $\pc{}=0.25$, and this conjecture is strongly supported by exact computations of Droz, Rey, Frachebourg
and Piasecki \cite{DRFP} resolving related differential equations. However, these results
are not entirely rigorous, and do not provide any intuitive understanding of the process.

Ballistic annihilation has more recently received significant attention from mathematicians, stemming from the popularisation
by Kleber and Wilson \cite{KW14} of a related ``bullet problem'', in which particles with independent uniformly-distributed random speeds 
leave the origin, and are annihilated by collisions. The distribution of the number of surviving bullets
from a given finite number of shots was rigorously established by Broutin and Marckert \cite{BM17},
using a surprisingly intricate argument. When infinitely many bullets are fired, it is conjectured that
there is some critical speed $s_{\mathrm{c}}>0$ such that the first bullet survives with positive probability if it has
a higher speed, but is almost surely annihilated if it has a lower speed. Dygert, Kinzel, Zhu, Junge, Raymond and Slivken \cite{DKZJRS}
solved the corresponding problem for speeds chosen uniformly from a given finite set, and their results also give upper
bounds for the ballistic annihilation problem.

A discretised version of ballistic annihilation was introduced recently by Burdinski, Gupta, and Junge \cite{BGJ18}. In this variation,
instead of random starting positions we start with one particle at every integer point. The principle difference this makes to the
process is that triple collisions occur with positive probability; in a triple collision, all three particles are annihilated.

An ergodicity argument implies that (at any value of $p$) each moving particle is eventually annihilated,
and so the basic question of ballistic annihilation is whether stationary particles can survive forever.
It is easy to see that for sufficiently large values of $p$ almost surely infinitely many stationary particles 
will survive. For example, if a stationary particle is to be annihilated, there must be an interval
containing that particle which contains at least as many moving particles as stationary particles at time $0$,
but for $p>0.5$ there is a positive probability that no such interval exists. 

We shall consider the function $\theta(p)$, being the probability that a given stationary particle, 
without loss of generality positioned at $0$, survives forever; we write $\psi(p)$ for the corresponding 
probability in the discretised model. Since it is not obvious that $\theta(p)$ is increasing, we cannot 
say that there is necessarily a single critical probability. Thus we write
\begin{gather*}\pc-:=\inf\{p\in [0,1]:\theta(p)>0\}\,,\\
\pc+:=\sup\{p\in [0,1]:\theta(p)=0\}\,;\end{gather*} 
clearly these both exist and $0\leq\pc-\leq\pc+\leq 0.5$. 
We write $\pc[\hat]-$ and $\pc[\hat]+$ for the corresponding values for the discretised process.

Improved upper bounds on $\pc+$ have recently been obtained independently by Dygert et al.\ \cite{DKZJRS}
(who prove $\pc+\leq0.3313$) and by Sidoravicius and Tournier \cite{ST17} (who prove $\pc+\leq 1/3$ and sketch details which improve their bound to
$\pc+\leq 0.32803$). Burdinski, Gupta, and Junge \cite{BGJ18} show that $\pc[\hat]+\leq0.287$. They argue that the heuristic rationale
for $\pc{}=0.25$, adapted to include the possibility of triple collisions, but only between consecutive particles, would suggest a phase transition
at about $0.245$, and conjecture that a critical probability for the discretised process exists and is slightly smaller than this value.

Despite these results, there have been no corresponding lower bounds, and in fact it was not previously known that stationary particles 
are almost surely annihilated for any nontrivial values of $p$. Proving that almost sure annihilation occurs for all
sufficiently small $p$ was therefore the central open question in ballistic annihilation \cite{BGJ18}. Our results not only
answer this question, but in fact give a lower bound which is closer to the conjectured critical probability than
the best known upper bounds.

\section{Results}
In order to analyse the survival of particles in ballistic annihilation, we shall instead consider the same process restricted 
to a finite or semi-infinite interval, with an absorbing barrier at each endpoint;
here we think of particles outside the interval being frozen. We make a slight distinction between open intervals, where
a particle starting at the endpoint is frozen, and any particle reaching the endpoint is absorbed rather than annihilated,
and closed intervals, where a particle starting at the endpoint is not frozen, and two particles may be annihilated at the endpoint.
Of course, this distinction does not matter when no particle starts at the endpoint; however, we shall frequently define the interval
in terms of the starting positions of the vertices.

For an interval $I$, we write $\{C\}_{I}$ for the event that condition $C$ is satisfied on the restricted process. The principle interval
we work with is $\future$, and so for conciseness we shall omit the subscript when referring to this interval.
When writing events, we use $\bullet$ for an arbitrary particle, $\go$, $\stay$ and $\come$ to indicate particles with a particular velocity, 
and $\to$ or $\from$ to show that a particle collides with another or is absorbed. We also write $\bullet_i$ for the $i$th particle in the right half-line 
(by starting position, from left to right), and $x_i$ for its position; in the discretised version we simply have $x_i=i$. We write $\prob{\cdot}$ for
probabilities in the original setting, and $\probz{\cdot}$ when dealing with the discretised setting. In the latter, we draw a distinction between single
and triple collisions, with \eg $\bullet_1\to\bullet$ indicating specifically a single collision, and being disjoint from $\bullet_1\triple$ 
(note that this differs from the way similar notation is used in \cite{BGJ18}).

In particular we consider the probability $q:=\prob{0\from\bullet}$, and the corresponding probability for the integer case 
$\hat{q}:=\probz{0\from\bullet}$. Note that 
\begin{equation}\theta(p)=(1-q)^2\,,\label{theta}\end{equation} 
and likewise
\begin{equation}\psi(p)=(1-\hat{q})^2\,.\label{psi}\end{equation} 
This is because with probability $q$ a particle from the right half-line would reach a stationary particle at $0$ if the left half-line were frozen, 
and independently with probability $q$ a particle from the left would reach $0$ if the right were frozen; if either of these events occurs then whichever 
particle started closest to $0$ will annihilate the original particle before particles from the other half line can prevent it. 
Thus proving (for a particular value of $p$) that almost surely every particle is eventually annihilated is equivalent to proving that $q=1$.

Consider the process on the right half-line $\future$, and reveal particles one by one from the left. Write $S_1$ for the event $\exists i: \onab{i}{\go[1]\to\stay[i]}$, 
that is, the event that the first particle is moving right, and, as we reveal further particles, we eventually find one which is stationary and which the first particle 
could hit, if all subsequent particles were frozen. Similarly, write $S_k$ for the event that, as we continue to reveal particles, we see at least $k$ such particles, that is
\[\exists i_1<\cdots<i_k: \bigcap\nolimits_{j=1}^k\onab{i_j}{\go[1]\to\stay[i_j]}\,.\]
Note that in order for $S_k$ to occur, $\bullet_{i_j}$ cannot be annihilated by $\bullet_1$ (in the overall process) for any $j<k$, 
and so is instead annihilated from the right. For each $k\geq 1$, let $s_k=\prob{S_k}$ and $\hat{s}_k=\probz{S_k}$.

A crucial ingredient in our proof will be the following result.
\begin{lemma}\label{flip}$\sum_{k\geq 1}s_k=pq$ and $\sum_{k\geq 1}\hat{s}_k=p\hat{q}$.\end{lemma}
\begin{proof}For a configuration $\omega\in S_k$, let $\rev(\omega,k)$ be the configuration obtained by reversing the interval $[x_1,x_{i_k}]$, that is, 
a particle at position $x$ in $\omega$ corresponds to a particle at position $x_{i_k}+x_1-x$ moving in the opposite direction in $\rev(\omega,k)$; note that
if $\omega$ is a valid configuration for the discretised model then so is $\rev(\omega,k)$.
In $\rev(\omega,k)$, the first particle is stationary, and the $i_k$th particle is moving to the left and will eventually collide with it. Thus 
$\rev(\omega,k)\in\{\stay[1]\from\bullet\}$.

Conversely, suppose $\omega^*\in\{\stay[1]\from\bullet\}$; clearly there is a unique $j>1$ for which 
$\omega^*\in\{\stay[1]\from\bullet_j\}$. Consequently, if $\omega^*=\rev(\omega,k)$ for some $\omega$ and $k$ then $\omega$ must be the 
configuration obtained from $\omega^*$ by reversing the interval $[x_1,x_j]$. Now $\omega\in\onab{j}{\go[1]\to\stay[j]}$, and hence there
is a unique $k$ such that $j$ is the $k$th value of $i$ for which $\omega\in\onab{i}{\go[1]\to\stay[i]}$. Thus for any 
$\omega^*\in\{\stay[1]\from\bullet\}$, there is a unique pair $(\omega,k)$ for which $\omega^*=\rev(\omega,k)$.

Since $\rev$ is a bijection between $\{(\omega,k):\omega\in S_k\}$ and $\{\stay[1]\from\bullet\}$, and is clearly measure-preserving, 
we have $\sum_{k\geq 1}s_k=\prob{\stay[1]\from\bullet}$ and $\sum_{k\geq 1}\hat{s}_k=\probz{\stay[1]\from\bullet}$. 
Note that the event $\{\stay[1]\from\bullet\}$ occurs if and only if
the events ${\stay[1]}$ and $\onfuture[x_1]{x_1\from\bullet}$ both occur, and these are independent 
(because they are defined on disjoint intervals).
Also, by translation invariance,
\[\probb{\onfuture[x_1]{x_1\from\bullet}}=\probb{\onfuture{0\from\bullet}}\,,\]
and the same is true for the discretised process, giving $\prob{\stay[1]\from\bullet}=pq$ and $\probz{\stay[1]\from\bullet}=p\hat{q}$, as required.
\end{proof}
Let $r$ be the probability, in the process restricted to $\future$, that the first particle is right-moving and is annihilated in a single collision 
with a stationary particle, and some particle reaches $0$, \ie
\[r:=\prob{(\go[1]\to\stay)\wedge(0\from\bullet)}\,,\]
and let $\hat{r}$ be the corresponding probability for the discretised process.
\begin{lemma}\label{conditioning}$q=\frac{1-p}{2}(1+q)+r(1-q)+pq^3$ and $\hat{q}=\frac{1-p}{2}(1+\hat{q})+\hat{r}(1-\hat{q})+p\hat{q}^3$.\end{lemma}
\begin{proof}
We prove the latter statement; the only changes required to prove the former are the omission of terms involving a triple collision and the substitution of 
$x_1$ for $1$, etc., as appropriate.

Conditioning on the velocity of the first particle, we have
\begin{equation}\hat{q}=\frac{1-p}{2}\probz{0\from\bullet\mid\come[1]}
+p\probz{0\from\bullet\mid\stay[1]}
+\probz{(0\from\bullet)\wedge(\go[1])}\,.\label{cond1}\end{equation}
Clearly if the first particle moves left it will reach $0$. If the first particle is stationary, it is annihilated with probability $\hat{q}$, since this
equals $\probbz{\onfuture[1]{1\from\bullet}}$. Note, however, that this event occurs if and only if $\onabz[1]{j}{\stay[1]\from\come[j]}$ occurs 
for some $j$, since the progress of a left-moving particle cannot be affected by particles further to the right. 
Given that $\onabz{j}{\stay[1]\from\come[j]}$ occurs,
a particle reaches $0$ if and only if $\onfuture[j]{j\from\bullet}$ occurs, since the fact that $\bullet_j$ is left-moving and annihilates
$\bullet_1$ means that no particle from the right of $\bullet_j$ can encounter any particles after reaching $j$. Clearly $\onfuture[j]{j\from\bullet}$
is independent of $\onabz{j}{\stay[1]\from\come[j]}$ and has probability $\hat{q}$, so $\probz{0\from\bullet\mid\stay[1]}=\hat{q}^2$.

If the first particle moves right, it must eventually be annihilated (see \eg Lemma 3.3 of \cite{ST17}). Thus we have
\begin{equation}\frac{1-p}{2}=\probz{\go[1]\to\stay}+\probz{\go[1]\to\come}+\probz{\go[1]\triple}\,.\label{hits}\end{equation}
Conditioning on how the first particle is annihilated, we have
\begin{equation}
\begin{split}\probz{(0\from\bullet)\wedge(\go[1])}={}&\probz{(0\from\bullet)\wedge(\go[1]\to\stay)}
+\probz{0\from\bullet\mid\go[1]\to\come}\probz{\go[1]\to\come}\\
&+\probz{0\from\bullet\mid\go[1]\triple}\probz{\go[1]\triple}\,.\end{split}\label{cond}
\end{equation}
Now $\probz{0\from\bullet\mid\go[1]\to\come}=\probz{0\from\bullet\mid\go[1]\triple}=\hat{q}$, since, 
given that $\come[j]$ annihilates $\go[1]$, either in a single or triple collision,
$\{0\from\bullet\}$ if and only if $\onfuture[j]{j\from\bullet}$. Thus \eqref{cond} becomes
\begin{equation}
\probz{(0\from\bullet)\wedge(\go[1])}=\probz{(0\from\bullet)\wedge(\go[1]\to\stay)}
+\hat{q}\probz{\go[1]\to\come}+\hat{q}\probz{\go[1]\triple}\,,\label{q-if-right}
\end{equation}
and, combining \eqref{hits} and \eqref{q-if-right},
\begin{align}
\probz{(0\from\bullet)\wedge(\go[1])}&=\hat{r}+\hat{q}\bbr{\frac{1-p}{2}-\prob{\go[1]\to\stay}}\nonumber\\
&=\hat{r}+\hat{q}\bbr{\frac{1-p}{2}-\hat{r}-\prob{(\go[1]\to\stay)\wedge(0\nfrom\bullet)}}\,.\label{use-r}\end{align}
To complete the proof, note that $(\go[1]\to\stay)\wedge(0\nfrom\bullet)$ occurs if and only if for some $k$ there are $k$ stationary particles, 
$\stay[i_1],\ldots,\stay[i_k]$, such that $\onabz{i_j}{\go[1]\to\stay[i_j]}$ occur, $\onabz{i}{\go[1]\to\stay[i]}$ does not occur for any other
value of $i<i_k$, and furthermore $\onfuture[i_k]{i_k\nfrom\bullet}$ occurs; 
this last because any particle which reaches $i_k$ from the right will either annihilate $\stay[i_k]$ before $\go[1]$ does, create a
triple collision with $\stay[i_k]$ and $\go[1]$, or reach $0$. Thus
\begin{align}
\probz{(\go[1]\to\stay)\wedge(0\nfrom\bullet)}&=\sum_{k\geq1}\hat{s}_k\probbz{\onfuture[i_k]{i_k\nfrom\bullet}\mid S_k}\nonumber\\
&=\sum_{k\geq1}\hat{s}_k(1-\hat{q})\nonumber\\
&=p\hat{q}(1-\hat{q})\,,\label{exact}
\end{align}
using the fact that $S_k$ is determined on the disjoint interval $[0,i_k]$ and the result of \lem{flip}. 
Combining \eqref{exact}, \eqref{use-r} and \eqref{cond1} gives
\begin{align*}\hat{q}&=\frac{1-p}{2}+p\hat{q}^2+\hat{r}+\hat{q}\bbr{\frac{1-p}{2}-\hat{r}-p\hat{q}(1-\hat{q})}\\
&=\frac{1-p}{2}(1+\hat{q})+\hat{r}(1-\hat{q})+p\hat{q}^3\,.\qedhere\end{align*}
\end{proof}
\lem{conditioning} immediately provides a nontrivial lower bound on the threshold for ballistic annihilation; although this
bound applies to both processes we state it for the discretised process, since we shall subsequently obtain stronger bounds
for the original process.
\begin{theorem}\label{fifth}If $p\leq 0.2$ then $\psi(p)=0$, and if $p>0.2$ then
\[\psi(p)\leq\frac{4p+1-3\sqrt{2p-p^2}}{2p}\,.\]\end{theorem}
\begin{proof}By \lem{conditioning}, $\hat{q}\geq\frac{1-p}{2}(1+\hat{q})+p\hat{q}^3$. 
Consider the function $f(x)=\frac{1-p}{2}(1+x)+px^3-x$, and note that
\[f(x)=(x-1)\bbr{px^2+px-\frac{1-p}{2}}\,.\]
Since $f(x)$ is continuous, $f(0)>0$ and $f(\hat{q})\leq 0$, $\hat{q}$ must be at least the smaller
positive root of $f(x)=0$, \ie
\[\hat{q}\geq\min\bbr{\frac{-p+\sqrt{2p-p^2}}{2p},1}\,.\] 
For $p\leq 1/5$ the smaller root is $1$, and so $\hat{q}=1$ and $\psi(p)=0$. For $p>1/5$ the other root is smaller,
and \eqref{psi} applied with this bound on $\hat{q}$ gives the required bound on $\psi(p)$.
\end{proof}
In order to improve the result of \thm{fifth} for the original process, we need a nontrivial lower bound on $r$.
\begin{lemma}\label{rbounds}$r\geq(1-p)pq/4$.\end{lemma}
\begin{proof}
We first claim that  
\begin{equation}(\go[1])\wedge(\stay[2])\wedge(x_3>2x_2-x_1)\wedge(\onfuture[2x_2-x_1]{(2x_2-x_1)\from\bullet})
\subset(\go[1]\to\stay)\wedge(0\from\bullet)\,.\label{quick}\end{equation}
To see this, note that if $(\go[1])\wedge(\stay[2])\wedge(x_3>2x_2-x_1)$ occurs then no other particle can reach $(0,x_2]$ before $\bullet_1$ 
annihilates $\bullet_2$, and if additionally $(\onfuture[2x_2-x_1]{(2x_2-x_1)\from\bullet})$ occurs then the particle which reaches $2x_2-x_1$
from the right cannot encounter any other particle before reaching $0$.

The velocities of particles are independent of each other and of their positions, and, for any given values of $x_1$ and $x_2$,
the event $(x_3>2x_2-x_1)$ depends only on the interval $(x_2,2x_2-x_1]$, disjoint from $\future[2x_2-x_1]$. Thus the four events 
on the left-hand side of \eqref{quick} are independent. Since they have probabilities $\frac{1-p}{2}$, $p$, $\frac12$ and $q$ 
respectively (the third because $x_2-x_1$ and $x_3-x_2$ are independent identically-distributed exponential random variables), the result follows.
\end{proof}
\begin{figure}
\centering
\includegraphics{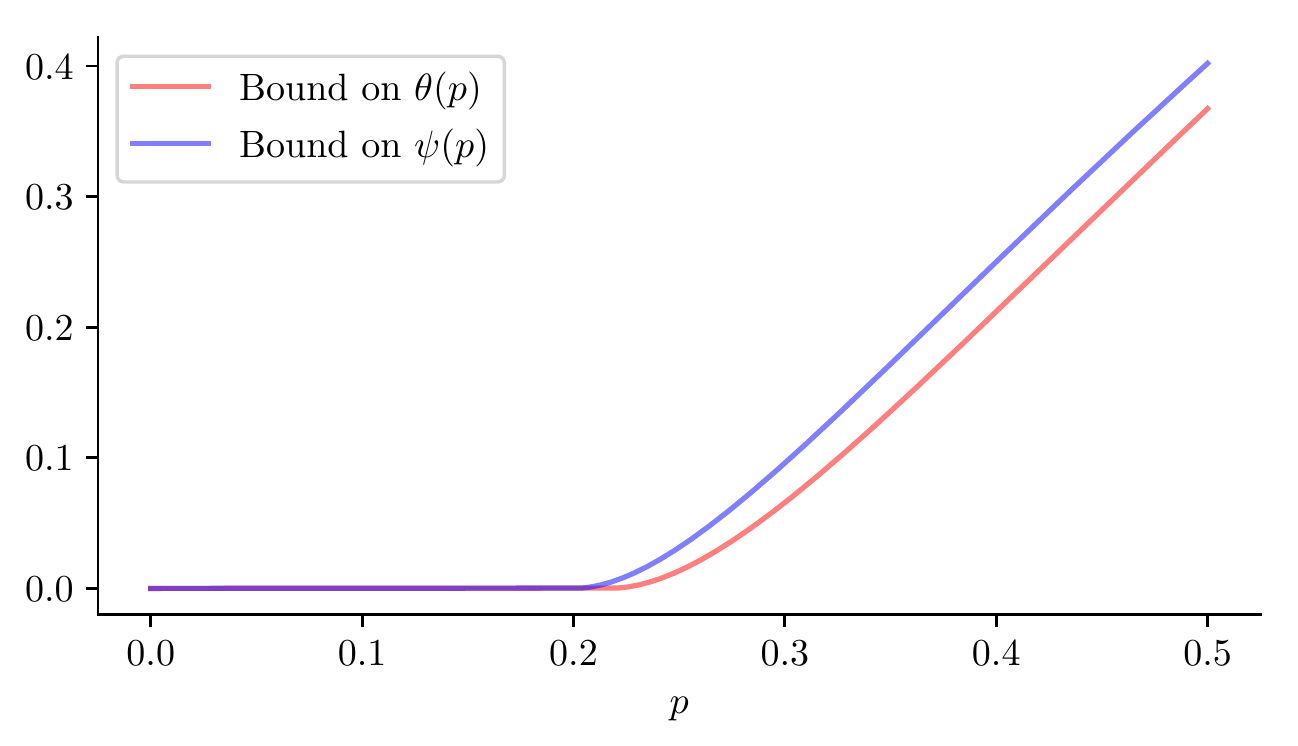}
\caption{Upper bounds on $\theta(p)$ and $\psi(p)$.}\label{thefigure}
\end{figure}
\begin{theorem}\label{better}If $p\leq\frac{\sqrt{89}-9}{2}$ then $\theta(p)=0$, and if $p>\frac{\sqrt{89}-9}{2}$ then
\[\theta(p)\leq\frac{16 + 49p + 14p^2 + p^3 - (11 + p)\sqrt{32p - 23p^2 + 6p^3 + p^4}}{32p}\,.\]
\end{theorem}
\begin{remark}$\frac{\sqrt{89}-9}{2}\approx 0.21699$.\end{remark}
\begin{proof}
\lem{conditioning} and \lem{rbounds} give $q\geq\frac{1-p}{2}(1+q)+(1-p)pq(1-q)/4+pq^3$. Writing
$g(x)=\frac{1-p}{2}(1+x)+(1-p)px(1-x)/4+px^3-x$, we have
\[g(x)=(x-1)\bbr{px^2+\frac{3p+p^2}{4}x-\frac{1-p}{2}}\,,\]
and $g(0)>0$ but $g(q)\leq 0$. Thus $q$ is at least the smallest positive root of $g(x)=0$, \ie
\[q\geq\min\bbr{\frac{-3p-p^2+\sqrt{32p - 23p^2 + 6p^3 + p^4}}{8p},1}\,.\] 
For $p\leq\frac{\sqrt{89}-9}{2}$ the smaller root is $1$, and so $q=1$ and $\theta(p)=0$. For $p>\frac{\sqrt{89}-9}{2}$ the other root is smaller,
and applying \eqref{theta} gives the required bound on $\theta(p)$.\end{proof}
Figure~\ref{thefigure} shows the bounds on $\theta(p)$ and $\psi(p)$ given by Theorems \ref{better} and \ref{fifth} respectively. Combining these with 
the results of \cite{ST17} and \cite{BGJ18} establishes that
\[0.21699\leq\pc-\leq\pc+\leq0.32803\]
and
\[0.2\leq\pc[\hat]-\leq\pc[\hat]+\leq0.287\,;\]
recall that it is conjectured that $\pc-=\pc+=0.25$ and $\pc[\hat]-=\pc[\hat]+<0.245$. Small improvements to these lower bounds could be made by more complicated
bounds on $r$ and $\hat{r}$ in the spirit of \lem{rbounds}. However, in order to obtain tight lower bounds it would be necessary to find exact expressions
for $r$ and $\hat{r}$ in terms of the other probabilities; in the discretised case any attempt to obtain a tight lower bound would be further 
hampered by the fact that there is as yet not even an exact conjecture for the best possible value.

\section*{Acknowledgements}
The author was supported by the European Research Council (ERC) under
the European Union's Horizon 2020 research and innovation programme (grant 
agreement no.\ 639046), and is grateful to Agelos Georgakopoulos for several
helpful discussions.

\end{document}